\documentclass[12pt]{amsart}
\usepackage{enumerate, amssymb}
\usepackage[matrix,arrow,curve]{xy}
\usepackage{mathrsfs}
\makeatletter \@addtoreset{equation}{section} \makeatother

\theoremstyle{plain}

\newtheorem{thm}[equation]{Theorem}
\newtheorem{cor}[equation]{Corollary}
\newtheorem{lem}[equation]{Lemma}
\newtheorem{prop}[equation]{Proposition}

\theoremstyle{definition}

\newtheorem{conv}[equation]{Convention}
\newtheorem{nota}[equation]{Notation}
\newtheorem{rem}[equation]{Remark}

\newtheorem{sit}[equation]{}

\newcommand{\res}{\operatorname{res}}

\newcommand{\Aut}{{\operatorname{Aut}}}

\newcommand{\supp}{\operatorname{supp}}
\newcommand{\mult}{\operatorname{mult}}

\newcommand{\affcone}{\operatorname{cone}}
\newcommand{\Pic}{\operatorname{Pic}}

\newcommand{\Spec}{\operatorname{Spec}}
\newcommand{\Sing}{\operatorname{Sing}}
\newcommand{\Bs}{\operatorname{Bs}}
\newcommand{\pr}{\operatorname{pr}}
\newcommand{\Div}{\operatorname{Div}}

\newcommand{\qq}{\mathbin{\sim_{\scriptscriptstyle{\Q}}}}
\newcommand{\A}{{\mathbb A}}
\newcommand{\PP}{{\mathbb P}}

\newcommand{\Q}{{\mathbb Q}}
\newcommand{\Z}{{\mathbb Z}}
\newcommand{\N}{{\mathbb N}}

\newcommand{\G}{{\mathbb G}}

\def\cL{{\mathcal L}}

\def\cL{{\mathscr L}}

\title[]{Unipotent group actions
on del Pezzo cones}

\author{Takashi Kishimoto}

\address{Department of Mathematics,
Faculty of Science, Saitama University, Saitama 338-8570, Japan}
\email{tkishimo@rimath.saitama-u.ac.jp}

\author{Yuri Prokhorov}

\address{Department
of Algebra, Faculty of Mathematics, Moscow State University,
Moscow 117234, Russia \quad and \quad
Laboratory of Algebraic Geometry, SU-HSE,
7 Vavilova Str., Moscow 117312, Russia
} \email{prokhoro@gmail.com}

\author{Mikhail Zaidenberg}

\address{Universit\'e
Grenoble I, Institute Fourier, UM 5582 CARS-UHF, B 74, 38402 St.\
Martin derri\`eres codex, France} \email{zaidenbe@ujf-grenoble.fr}

\thanks{
The first author was supported by a Grant-in-Aid for Scientific
Research of JSPS No. 24740003. The second author was partially
supported by RFBR grants No. 11-01-00336-a, the grant of Leading
Scientific Schools No. 4713.2010.1, Simons-IUM fellowship,
 and
AG Laboratory SU-HSE, RF government
grant ag. 11.G34.31.0023.
}

\begin{document}

\begin{abstract}
In our previous paper \cite{Kishimoto-Prokhorov-Zaidenberg} we
showed that for any del Pezzo surface $Y$ of degree $d \geqq 4$
and for any $r \geqq 1$, the affine cone
$X={\affcone}_{r(-K_Y)}(Y)$ admits an effective $\G_a$-action. In
particular, the group $\Aut(X)$ is infinite dimensional. In this
note we prove that for a del Pezzo surface $Y$ of degree $\le 2$
the generalized cones $X$ as above do not admit any non-trivial
action of a unipotent algebraic group.
\end{abstract}

\subjclass[2010]{Primary 14R20, 14J45; \ Secondary 14J50, 14R05 }
\keywords{affine cone, del Pezzo surface, additive group, group action}
 \maketitle

\section{Introduction}\label{-section-Introduction}

We are working over an algebraically closed field $\Bbbk$ of
characteristic $0$. Let $Y$ be a smooth projective variety with a
polarization $H$, where $H$ is an ample Cartier divisor. A
\emph{generalized affine cone} over $(Y,H)$ is the normal affine
variety
$$\affcone_{H}(Y)=\Spec \bigoplus_{\nu\ge 0} H^0(Y,\nu H)\,.$$ This variety
$\affcone_{H}(Y)$ is the usual affine cone over $Y$ embedded in a
projective space  $\PP^n$ by the linear system $|H|$ provided that
$H$ is very ample and the image of $Y$ in $\PP^n$ is projectively
normal.

In this paper we deal with a del Pezzo surface $Y$ and a
pluri-anticanonical divisor $H=-rK_Y$ on $Y$, where $r\ge 1$; we
call then $\affcone_{H}(Y)$ a {\em del Pezzo cone}. This is a
usual cone if $r\ge 4-d$ (see e.g. \cite[Theorem
8.3.4]{Dolgachev-topics}) and a generalized cone otherwise.

It is known \cite[3.1.13]{Kishimoto-Prokhorov-Zaidenberg} that for
any smooth rational surface there is an ample polarization such
that the associated affine cone admits an effective
${\G}_a$-action. Furthermore, for any del Pezzo surface of degree
$\geqq 4$ the corresponding del Pezzo cones $\affcone_{-rK_Y}(Y)$
($r\ge 1$) admit such an action ({\em loc.cit}). The latter holds
also for some smooth rational Fano threefolds with Picard number 1
\cite{Kishimoto-Prokhorov-Zaidenberg,
Kishimoto-Prokhorov-Zaidenberg-2011}. However, for  del Pezzo
surfaces of small degrees the consideration turns out to be more
complicated. It is unknown so far whether the affine cone over a
smooth cubic surface in $\PP^3$ admits a $\G_a$-action (cf.\
\cite[\S 4]{Kishimoto-Prokhorov-Zaidenberg}). In this paper we
investigate the cases $d=1$ and $d=2$. Our main result can be
stated as follows.

\begin{thm}\label{mainthm}
Let $Y$ be a del Pezzo surface of degree $d={K_Y}^2\le 2$. Then
for any $r\ge 1$ there is no non-trivial action of a unipotent
group on the generalized affine cone
\[
X_r= \affcone_{-rK_Y}(Y)=\Spec A,\quad \text{where}\quad A=
\bigoplus_{\nu\ge 0} H^0(Y,-\nu rK_Y)\,.
\]
\end{thm}

\begin{cor} In the notation as before assume that $d\le 2 $ and
$r\ge 4-d$ so that $X_r=\affcone_{-rK_Y}(Y)$ is a usual del Pezzo
cone. Then any algebraic subgroup $G\subset \Aut (X_r)$ is
isomorphic to a subgroup of $\G_m\times \Aut(Y)$, where $\Aut(Y)$
is finite.
\end{cor}

\begin{proof}  As follows from Theorem \ref{mainthm} $G$ is a reductive group.
Thus by Lemma 2.3.1 and Proposition 2.2.6 in
\cite{Kishimoto-Prokhorov-Zaidenberg} there are an injection and an isomorphism
$$G\hookrightarrow {\rm Lin}(X_r)\simeq
\G_m\times {\rm Lin}(Y)\subset \G_m\times\Aut(Y)\,,$$
where the  group $\Aut(Y)$ is finite, see \cite{Dolgachev-topics}.\end{proof}

We suggest the following

\begin{sit} {\bf Conjecture.} {\em If $d\le 2 $ then for any $r\ge d-4$
the full automorphism group $\Aut (X_r)$ is a finite extension of
the multiplicative group $\G_m$.}\end{sit}

Likewise in
\cite{Kishimoto-Prokhorov-Zaidenberg-2011, Kishimoto-Prokhorov-Zaidenberg}
we use a geometric criterion
of existence of an effective
$\G_a$-action on the affine cone $\affcone_H (Y)$
(see \cite{Kishimoto-Prokhorov-Zaidenberg-criterion} and Theorem \ref{crit} below).

Sections \ref{-section-Preliminaries}, \ref{-section-del-Pezzo},  and
 \ref{-section-cylinders}
contain necessary preliminaries. Theorem \ref{mainthm} is proven
in section \ref{section-proof}. The proof proceeds as follows.
Assuming to the contrary that there exists a non-trivial unipotent
group action on $X_r=\affcone_{(-rK_Y)}(Y)$, there also exists an
effective $\G_a$-action on $X_r$. By Theorem \ref{crit} there is
an effective $\Q$-divisor $D$ on $Y$ such that $D \qq -K_Y$ and
$U=Y\setminus  D\cong Z \times \A^1$, where $Z$ is a smooth
rational affine curve. Such a principal open subset $U$ is called
in \cite{Kishimoto-Prokhorov-Zaidenberg} a {\em $(-K_Y)$-polar
cylinder}. One of the key points consists in an estimate for the
singularities of the pair $(Y,D)$. More precisely, we consider the
linear pencil $\cL$ on $Y$ generated by the closures of the fibers
of the projection $U\cong Z \times \A^1 \to Z$. Letting $S$ be the
last exceptional divisor appearing in the process of the minimal
resolution of the base locus of $\cL$ we compute the discrepancy
$a(S; D)$. Using this and some subtle geometrical properties of
the pair $(Y,D)$ we finally come to a contradiction.

\section{Criterion}\label{-section-Preliminaries}
Let $Y$ be a projective variety and $H$ be an ample divisor on
$Y$. Recall \cite{Kishimoto-Prokhorov-Zaidenberg} that an {\it
$H$-polar cylinder} in $Y$ is an open subset $U=Y\setminus \supp
(D)$ isomorphic to $Z \times \A^1$ for some affine variety $Z$,
where $D=\sum_i \delta_i \Delta_i$ with $\delta_i>0$ \ $\forall i$
is an effective $\Q$-divisor on $Y$ such that $q D$ is integral
and $qD\sim H$ for some $q\in\N$. Corollary 2.12 in
\cite{Kishimoto-Prokhorov-Zaidenberg-criterion}\footnote{Cf.\ also
\cite[3.1.9]{Kishimoto-Prokhorov-Zaidenberg}.} provides  the
following useful criterion of existence of an effective
$\G_a$-action on the affine cone.

\begin{thm}\label{crit}
Let $Y$ be a normal projective algebraic variety
with an ample polarization $H\in\Div (Y)$,
and let $X=\affcone_H (Y)$ be the
corresponding generalized affine cone. If $X$ is normal then
$X$ admits an effective $\G_a$-action if and only
if $Y$ contains an $H$-polar cylinder.
\end{thm}

We apply this criterion to a del Pezzo surface $Y$ of degree
$d\le 2$ and a generalized cone $$X_r=\Spec \bigoplus_{\nu\ge 0}
H^0(Y,-\nu rK_Y)$$  associated with $H=-rK_Y$, where $r\ge 1$. It
follows, in particular, that if the cone $X_r$ admits an effective
$\G_a$-action then $Y$ contains a cylinder $Y\setminus\supp D$
with $qD\sim -rK_Y$. Hence $\frac{q}{r}D\sim_\Q -K_Y$. Replacing
$D$ by $\frac{q}{r}D$ we assume in the sequel that $D\sim_\Q-K_Y$.
This assumption leads finally to a contradiction, which proves
Theorem \ref{mainthm}.

\section{Preliminaries on weak del Pezzo surfaces}\label{-section-del-Pezzo}
A smooth projective surface $Y$ is called a {\em del Pezzo
surface} if the anticanonical divisor $-K_Y$ is ample, and a {\em
weak del Pezzo surface} if $-K_Y$ is big and nef. The \emph{degree} of
such a surface is $\deg Y=K_Y^2\in\{1,\ldots,9\}$.

\begin{lem}[see e.g.\ {\cite[Proposition 8.1.23]{Dolgachev-topics}}]
\label{WDP}
Blowing up a point on a del Pezzo surface
 of degree $d\ge 2$ yields a weak del Pezzo surface of degree $d-1$.
\end{lem}

\begin{thm}[see e.g. {\cite[Thm. 8.3.2]{Dolgachev-topics}}]
\label{theorem-weak-del-Pezzo}
Let $Y$ be a  del Pezzo surface of degree $d$. Then the following hold.
\begin{enumerate}
\item
If $d \ge 3$ then $|-K_Y|$ defines an embedding $Y\hookrightarrow \PP^d$.

\item
If $d = 2$  then $|-K_Y|$ defines a double cover $\Phi: Y\to \PP^2$
branched along a smooth curve $B\subset \PP^2$ of degree $4$.

\item
If $d = 1$ then $|-K_X|$ is a pencil with a single base point, say $O$.
The linear system $|-2K_Y|$ defines a double cover
$\Phi: Y\to Q'\subset \PP^3$, where $Q'$ is a quadric cone with vertex at
$\Phi(O)$. Furthermore $\Phi$ is  branched along a smooth curve
$B\subset Q'$ cut out on $Q'$ by a cubic surface.
\end{enumerate}
\end{thm}

The  Galois involution $\tau : Y \to Y$ associated to the double
cover $\Phi$ is a regular morphism. It is called \emph{Geiser
involution} in the case $d=2$ and  \emph{Bertini involution} in
the case $d=1$.

\begin{rem}\label{remark-(-1)-curve} Recall the following
facts (see e.g. \cite{Dolgachev-topics}).
For an irreducible curve $C$ on $Y$ we have $C^2\ge-1$
if $Y$ is a del Pezzo surface and $C^2\ge-2$ if $Y$  is a
weak del Pezzo surface. In both cases $C^2=-1$ if and only if
$C$ is a $(-1)$-curve, if and only if $-K_Y\cdot C=1$, and
$C^2=-2$ if and only if $C$ is a $(-2)$-curve, if and only if
$-K_Y\cdot C=0$. A weak del Pezzo surface is del Pezzo
if and only if it has no $(-2)$-curve.

If  $d\ge 2$
then any curve $C$ on $Y$
such that $-K_Y\cdot C=1$ is an irreducible smooth rational curve by
(i) and (ii). By the adjunction formula  such $C$
must be a $(-1)$-curve.
\end{rem}

\begin{lem}\label{claim}
Let $Y$ be a del Pezzo surface of degree $d\le 2$. Then any member
$R\in |-K_{Y}|$ is reduced and $p_{a}(R)=1$. Moreover, $R$ is
irreducible except in the  case where
\begin{itemize}
\item $d=2$, $R=R_{1}+R_{2}$, $R_{i}^2=-1$, $i=1,2$,
$R_{1}\cdot R_{2}=2$, and $R_2=\tau(R_1)$.
\end{itemize}
Furthermore, $\operatorname{Sing}(R)\subset \Phi^{-1}(B)$ and for
any $P\in \Phi^{-1}(B)$ there is a unique member $R\in |-K_Y|$
singular at $P$.
\end{lem}

\begin{proof} We have $p_{a}(R)=1$ by
adjunction. Let $R_1\varsubsetneq R$ be a reduced irreducible
component. Then $(-K_{Y})\cdot R_1<(-K_{Y})\cdot R=d$ and so $d=2$
and $R_1$ is a $(-1)$-curve by Remark \ref{remark-(-1)-curve}.
Since $R^2=d=2$, $R\neq 2R_{1}$. Therefore $R=R_1+R_2$, where
the $R_i$ ($i=1,2$) are $(-1)$-curves and $R_{1}\cdot
R_{2}=\frac12 (R^2-R_1^2-R_2^2)=2$. Finally, in both cases we have
$R=\Phi^{-1}(L)$, where $L$ is a line  in $ \PP^2$. Thus $R$
is singular at $P$ if and only if $\Phi(P)\in B$ and $L$ is
tangent to $B$ at $\Phi(P)$.
\end{proof}

\begin{rem}\label{reamrk-Geiser}
Let $R_1$ and $R_2$ be $(-1)$-curves on  a del Pezzo surface $Y$
of degree $2$
such that $R_1\cdot R_2\ge 2$. Then $R_2=\tau(R_1)$, $R_1\cdot R_2=2$,
and  $R_1+R_2\in | -K_Y|$.
Indeed, $R_1+\tau(R_1)\sim -K_Y$. Hence
$\tau(R_1)\cdot R_2=-1$ and so $\tau(R_1)= R_2$.
\end{rem}

\section{$(-K)$-polar cylinders on del Pezzo surfaces}\label{-section-cylinders}
We adjust here some lemmas in \cite[\S 4]{Kishimoto-Prokhorov-Zaidenberg}
to our setting.
\begin{nota}\label{notation-0026}
Let $Y$ be a del Pezzo surface of degree $d$. Suppose that
$Y$ admits a $(-K_Y)$-polar cylinder
\begin{equation}\label{equation-001}
U=Y \setminus \supp(D) \cong Z \times \A^1, \quad \text{where}\quad\
D=\sum_{i=1}^n \delta_i \Delta_i \qq -K_Y \quad (\delta_i>0)
\end{equation}
and $Z$ is a smooth rational affine curve.
We let $\cL$ be the linear pencil on $Y$ defined
by the rational map $\Psi: Y\dashrightarrow \PP^1$ which extends the
projection $\pr_1:U\cong Z\times\A^1\to Z$.

Resolving, if necessary, the base locus of the pencil $\cL$
we obtain a  diagram
\begin{equation}\label{equation-001dia}
\xymatrix{
&W\ar[dl]_p\ar[dr]^q&
\\
Y\ar@{-->}[rr]^{\Psi}&&\PP^1
}
\end{equation}
where we let $p:W \to Y$ be the shortest succession
of blowups such
that the proper transform $\cL_W:=p_*^{-1} \cL$ is base point free.
Let $S$ be the last exceptional
curve of the modification $p$ unless $p$ is the identity map,
i.e., $\Bs \cL =\emptyset$.
Notice that $S$ is a unique $(-1)$-curve in the exceptional
locus $p^{-1}(P)$ and a section of $q$. The restriction
${\Phi_{\cL_W}}|_U$ is an $\A^1$-fibration and its fibers are reduced,
irreducible affine curves with one place at infinity, situated on
$S$.
\end{nota}

\begin{lem}\label{lem2}
One of the following holds.
\begin{enumerate}
 \item
$\Bs \, \cL$ consists of a single point, say $P$;
 \item
$\Bs \, \cL=\emptyset$ and $5\le d\le 8$.
\end{enumerate}
\end{lem}

\begin{proof}
Since the general
members of $\cL$ are disjoint in $U$ and each one meets the
cylinder $U$ along an $\A^1$-curve, $\Bs \cL $ consists of at most
one point, which we denote by $P$.
Suppose that $\Bs \cL =\emptyset$. Then
the pencil $\cL$ yields a conic bundle $\Psi:Y \to \PP^1$
with a section, which is a component of $D$, say
$\Delta_0$.
In particular $d\le 8$.
For a general fiber $L$ of $\Psi$ we have
$$
L^2=0,\qquad -K_Y\cdot L=2=D\cdot L=\delta_0.
$$
Note that $\Psi$ has exactly $8-d$ degenerate fibers $L_1,\dots, L_{8-d}$.
Each of these fibers  is reduced and consists of two $(-1)$-curves
meeting transversally at a point.
Let $C_i$ be the component of $L_i$ that
meets $\Delta_0$.
We claim that each $C_i$ is a component of $D$.
Indeed, otherwise
\[1=
-K_Y\cdot C_i=D\cdot C_i\ge \delta_0\Delta_0\cdot C_i=\delta_0=2\,,
\]
a contradiction.
Therefore we may assume that $C_i=\Delta_i$ and so
\[
1=D\cdot C_i\ge \delta_0\Delta_0\cdot C_i+ \delta_i C_i^2= 2-\delta_i\,.
\] Hence $\delta_i\ge 1$, $i=1,\ldots,8-d$.
We obtain
\[
d=-K_Y\cdot D\ge \sum \delta_i\ge \delta_0+ \sum_{i=1}^{8-d} \delta_i\ge
2+ 8-d=10-d\,.
\]
Thus $d\ge 5$ as stated.
\end{proof}

\begin{rem}
If $\Bs \cL=\{P\}$ ($\Bs \cL=\emptyset$, respectively) then all
components of $D$ (all components of $D$ except for $\Delta_0$,
respectively) are contained in the fibers of $\Psi$. Indeed,
otherwise not all the fibers of $\Psi|U$ were $\A^1$-curves,
contrary to the definition of a cylinder.
\end{rem}

\begin{lem}\label{picard}
The number of irreducible components of the
reduced curve $\supp (D)$, say $n$, is greater than or equal to $10-d$.

\end{lem}

\begin{proof}
Consider the exact sequence
\[
\bigoplus_{i=1}^n \Z [\Delta_i] \longrightarrow \Pic (Y)
\longrightarrow \Pic (U) \longrightarrow 0.
\]
Since $\Pic(Z)=0$ and $U \cong Z \times \A^1$ we have
$\Pic(U)=0$. Hence $n\ge \rho(Y)=10-d$, as stated.
\end{proof}

\begin{lem}\label{lemma-degenerate-fibers}
Assume that $\Bs \cL=\{P\}$.
Let $L$ be a member of  $\cL$ and $C$ be
an irreducible component of $L$.
Then the following hold.
\begin{enumerate}
\item
$\supp(L)$ is simply connected and $\supp(L)\setminus\{P\}$ is an $SNC$
divisor;
\item
$C$ is rational and smooth outside $P$;
\item
if $P\in C$ then $C \setminus \{P \}\simeq \A^1$.
\end{enumerate}
\end{lem}

\begin{proof}
All the assertions follow from the fact that $q$ in \eqref{equation-001dia}
is a rational curve fibration
and the exceptional locus of $p$ coincides with $p^{-1}(P)$.
\end{proof}

In the next lemma we study the singularities of the pair
$(Y,D)$. We refer to \cite{Kollar-1995-pairs} or to
\cite[Chapter 2]{Kollar-Mori-19988}
for the  standard terminology on singularities of pairs.

\begin{lem}[Key Lemma]\label{lem4}
Assume that $\Bs \cL=\{P\}$. Then the pair $(Y, D)$ is not log canonical
at $P$. More precisely, in notation as in \textup{\ref{notation-0026}} the
discrepancy $a(S; D)$ of $S$ with respect to $K_Y+D$ is equal to $-2$.
\end{lem}

\begin{proof}
We write
\begin{equation}\label{discr}
K_W +D_W \qq p^*(K_Y+D) +
a(S; D)S + \sum a(E; D)E,
\end{equation}
where the summation on the
right hand side ranges over the components of the
exceptional divisor of $p$ except for
$S$, and $D_W$ is the proper transform of $D$ on $W$. Letting $l$
be a general fiber of $q$, by \eqref{discr} we obtain
$$
-2=(K_W+D_W)\cdot l=a(S; D).
$$
Indeed,  $K_Y+D \qq 0$ and $l$ does not meet the curve
$\supp(D_W+p^{*}(P)-S)$.
This proves the assertion.
\end{proof}

\begin{cor}\label{cor1}
If $\Bs \cL=\{P\}$ then $\mult_P (D) >1$.
\end{cor}

\begin{proof}
Indeed, otherwise the pair $(Y,D)$ would be canonical by
\cite[Ex.\ 3.14.1]{Kollar-1995-pairs}, and in particular,
log canonical at $P$,
which contradicts Lemma \ref{lem4}.
\end{proof}

\begin{cor}\label{cor00121}
If $\Bs \cL=\{P\}$ then
every $(-1)$-curve $C$ on $Y$ passing
through $P$ is contained in $\supp(D)$.
\end{cor}

\begin{proof}
Assume to the contrary that $C$ is not a component of $D$.
Then
$$
\mult_P D\le C\cdot D=-K_Y\cdot C=1,
$$
which
contradicts Corollary \ref{cor1}.
\end{proof}

\begin{conv}\label{notation-D}
From now on we assume that $d\le 3$.
By Lemma \ref {lem2} we have $\Bs \cL=\{P\}$.
\end{conv}

\begin{lem}\label{lem0} We have $\lfloor D\rfloor =0$ i.e. $\delta_i< 1$
for all $i=1,\ldots,n.$ \end{lem}

\begin{proof} For the case $d=3$ see
\cite[Lemma 4.1.5]{Kishimoto-Prokhorov-Zaidenberg}.
Consider the case $d=1$. By Lemma \ref{picard} $n\ge 9$. For any
$i=1,\ldots,n$ we have
$$1=-K_Y\cdot D=\sum_{j=1}^n\delta_j
(-K_Y)\cdot \Delta_j>\delta_i (-K_Y)\cdot \Delta_i\,.$$ Since the
anticanonical divisor $-K_Y$ is ample, it follows that
$\delta_i<1$, as required.

Let further $d=2$. Assuming that $\delta_1\ge 1$
we obtain:
\begin{equation}\label{202}
2=-K_Y\cdot
D=\sum_{i=1}^n\delta_i(-K_Y)\cdot\Delta_i>\delta_1(-K_Y)\cdot
\Delta_1\ge
-K_Y\cdot\Delta_1\,,
\end{equation} where $n\ge 8$ by Lemma \ref{picard}.
It follows that $-K_Y\cdot\Delta_1=1$, i.e. $\Delta_1$ is a
$(-1)$-curve. Then $C:=\tau(\Delta_1)$ is also a $(-1)$-curve,
where $\tau$ is the Geiser involution, and $\Delta_1+C\sim -K_Y$.
If $C\subset\supp(D)$, e.g. $C=\Delta_2$, then by \eqref{202} we
obtain that $\delta_2<1$. Now $\Delta_1+\Delta_2\qq D$ yields a
relation with positive coefficients
$$
(1-\delta_2)\Delta_2\qq (\delta_1-1)\Delta_1+
\sum_{i=3}^n \delta_i\Delta_i.
$$
This implies that $C^2=\Delta_2^2\ge 0$,
a contradiction.

Hence $C\not\subset\supp(D)$. Thus $C\qq D-\Delta_1$, where the
right hand side is effective. This leads to a contradiction as
before.
\end{proof}

\begin{lem}\label{lem3}
\footnote{Cf.\ {\cite[Lemma 4.1.6]{Kishimoto-Prokhorov-Zaidenberg}}.}
 For a member $L$ of $\cL$, any irreducible component of $L$
passes through the base point $P$ of $\cL$.
\end{lem}

\begin{proof}
Assume to the contrary that there exists a component $C$ of $L$
such that $P\not\in C$. Then clearly ${C}^2<0$
 (see the proof of Lemma \ref{lem2}). Since also
$-K_Y\cdot C>0$, $C$ is a $(-1)$-curve. Let $C'$ be a component
of $L$ meeting $C$. If $P \not \in C'$, then $C$ and $C'$ are
both $(-1)$-curves and so $L=C+C'$. Thus $\cL=|C+C'|$ is base point
free, which contradicts Lemma \ref{lem2}. Hence $C'$ passes
through $P$. Since $P$ is a unique base point of $\cL$, $C$ does
not meet any member $L'\in \cL$ different from $L$.
By Lemma \ref{lemma-degenerate-fibers} $L$ is simply connected,
so  $C'$ is the only component of $L$ meeting $C$.
Note that $\supp(D)$ is connected because $D$ is ample.
Hence
$C'$ must be contained in $\supp(D)$. In fact,
supposing to the contrary that $C'$ is not contained in
$\supp (D)$, the curve $C$ must be contained in $\supp (D)$. Indeed,
the affine surface $U=Y\backslash \supp (D)$ does not contain
any complete curve. Since $\supp (D)$ is connected
there is an irreducible component
of $\supp (D)$ intersecting $C$ and passing through
$P$.
This contradicts Lemma \ref{lemma-degenerate-fibers}.
Thus we may suppose that $C'=\Delta_1$.

If $C\subset\supp(D)$, say, $C=\Delta_2$, then
$$
1=-K_Y \cdot
C=\left(\sum_{i=1}^n \delta_i \Delta_i\right) \cdot
\Delta_2=\delta_1 -\delta_2.
$$
Hence $\delta_1=\delta_2 +1> 1$,
which contradicts Lemma \ref{lem0}.

Therefore $C\not\subset\supp(D)$ and so
$$
1=-K_Y\cdot C= \left(\sum_{i=1}^n \delta_i \Delta_i\right)
\cdot C=\delta_1,
$$
which again gives a contradiction by Lemma \ref{lem0}.
\end{proof}

\section{Proof of Theorem \ref{mainthm}}\label{section-proof}
According to our geometric criterion
\ref{crit}, Theorem \ref{mainthm}
is a consequence of the following proposition.

\begin{prop}\label{yuta1}
Let $Y$ be a del Pezzo surface of
degree $d\le 2$. Then $Y$ does not admit any $(-K_Y)$-polar cylinder.
\end{prop}

\begin{conv}
We let $Y$ be a del Pezzo
surface of degree $d\le 2$. We assume to the contrary that $Y$
possesses a $(-K_Y)$-polar cylinder $U$ as in (\ref{equation-001}).
By Lemma \ref{lem2} we have $\Bs \cL=\{P\}$.
\end{conv}

\begin{lem}
\label{lemma-e}
For any $R\in |-K_Y|$ we have
$\supp(R)\nsubseteq\supp(D)$.
\end{lem}

\begin{proof}
Suppose to the contrary that $\supp(R)\subset\supp(D)$.
Let $\lambda\in \Q_{>0}$ be maximal such that $D-\lambda R$ is effective.
We can
write
\[
D= \lambda R+ D_{\res}\,,
\] where
$D_{\res}$ is an effective $\Q$-divisor
such that $\supp(R) \nsubseteq\supp(D_{\res})$.
For $t\in\Q_{\ge 0}$ we consider the following linear combination
\begin{equation*}
D_t:= D-tR + \frac{t}{1-\lambda}D_{\res}\qq -K_Y.
\end{equation*}
We have $D_0=D$ and $D_\lambda=\frac{1}{1-\lambda}D_{\res}$.
For $t<\lambda$, the $\Q$-divisor $D_t$ is effective
with $\supp(D_t)= \supp(D)$.
By Lemma \ref{lem4} applied to $D_t$ instead of $D$,
for any $t <\lambda$ the pair $(Y, D_t)$ is not
log canonical at $P$, with discrepancy $a(S; D_t)=-2$. Since
the function $t \mapsto a(S; D_t)$ is continuous, passing to the
limit we obtain $a(S; D_{\lambda})=-2$. Hence the pair
$(Y,D_{\lambda})$ is not log canonical at $P$ either and so
$\mult_P (D_\lambda)>1$.

Assume that $R$ is irreducible.
Since $R\subset \supp(D)$, $R$ is a component of
a member of $\cL$. Hence the curve $R$ is smooth outside $P$  and rational
(see Lemma \ref{lemma-degenerate-fibers}(ii)).
Since $p_a(R)=1$, $R$ is singular at $P$ and
$\mult_P (R)=2$. Since $R$ is different from the components of
$D_\lambda$ and $\mult_P (D_\lambda)>1$ we obtain
\begin{equation}
\label{equation-2}
2\ge K_Y^2= D_\lambda\cdot R\ge \mult_P (D_\lambda)\mult_P (R)>2,
\end{equation}
a contradiction.

Let further $R$ be reducible.
By Lemma \ref{claim} we have $d=2$ and $R=R_1+R_2$, where, say, $R_i=\Delta_i$,
$i=1$, $2$, are $(-1)$-curves passing through $P$ (see Lemma \ref{lem3}).
We may assume that $\delta_1\le \delta_2$ and so
$\lambda=\delta_1$.
Since $\Delta_1$ is not a component of $D_\lambda$ we obtain
\[
1= -K_Y\cdot R_1=D_\lambda\cdot \Delta_1\ge \mult_P
(D_\lambda)>1,
\]
a contradiction. This finishes the proof.
\end{proof}

\begin{proof}[Proof of Proposition \textup{\ref{yuta1}} in the case $d=1$.]
Since $\dim |-K_Y|=1$ there is $C\in |-K_Y|$ passing through $P$.
Furthermore, by Lemma \ref{claim}
 $C$ is
irreducible.
By Lemma \ref{lemma-e} $C$ is not contained in $\supp(D)$.
Likewise in \eqref{equation-2} we get  a contradiction. Indeed,
by Corollary \ref{cor1} we have
$$1=C^2=D\cdot C\ge\mult_PD\cdot\mult_PC>1\,.$$
\end{proof}

\begin{conv}
We assume in the remaining part that $d=2$.
\end{conv}

\begin{lem}\label{222}
A member $R\in |-K_Y|$ cannot be singular at $P$.
\end{lem}

\begin{proof}
Assume that $P\in \Sing(R)$. By Lemma \ref{claim} we have two
possibilities for $R$. Suppose first that $R$ is irreducible. By
Lemma \ref{lemma-e} $R\not\subset\supp(D)$ and we get a
contradiction likewise in \eqref{equation-2}. In the second case
$R=R_1+R_2$, where $R_1$ and $R_2$ are $(-1)$-curves passing
through $P$. Hence $R_1,\, R_2\subset\supp(D)$ by Corollary
\ref{cor00121}. The latter contradicts Lemma \ref{lemma-e}.
\end{proof}

\begin{nota}
We let $f: Y'\to Y$ be the blowup of $P$ and  $E'\subset Y'$
be the exceptional
divisor. By Lemma \ref{WDP} $Y'$ is a weak del Pezzo surface of degree $1$.
\end{nota}

\begin{sit}\label{newsit}  Applying Proposition \ref{yuta1} with $d=1$,
we can conclude that $Y'$ is not del Pezzo because it contains a
$-K_Y$-polar cylinder.
Indeed, let $D'$ be the crepant
pull-back of $D$ on $Y'$, that is,
\[
K_{Y'}+D'=f^*(K_Y+D)\qquad\text{and}\qquad f_*D'=D\,.
\] Then \begin{equation}\label{begeq}
D'=\sum_{i=1}^6 \delta_i\Delta_i'+\delta_0E',\quad\mbox{where}
\quad \delta_0=\mult_P(D)-1>0\,\end{equation} (see Lemma \ref{cor1})
and $\Delta'_i$ is the proper transform of $\Delta_i$ on $Y'$.
Thus $D'$ is an effective $\Q$-divisor on $Y'$ such that
$D'\sim_\Q-K_{Y'}$ and $Y'\setminus\supp
D'\simeq U\simeq Z\times\A^1$ is a $-K_Y$-polar cylinder. \end{sit}

\begin{lem}\label{corollary-floor-D}
We have $\mult_P(D)<2$ and
$\lfloor D'\rfloor=0$.
\end{lem}

\begin{proof}
Suppose first that all components of $D$ are $(-1)$-curves. Then
$\Delta_i\cdot \Delta_j=1$ for $i\neq j$ by Remark \ref{reamrk-Geiser}
and Lemma \ref{lemma-e}.
Hence $f$ is a log resolution of the pair $(Y,D)$. Therefore
$1-\sum\delta_i=a(Y,E')<-1$ by Lemma \ref{lem4}, so $\sum\delta_i>2$.
On the other hand
$2=-K_Y\cdot D=\sum\delta_i$, a contradiction. This shows that
there
exists a component $\Delta_i$ of $D$ which is not a $(-1)$-curve.
By the dimension count there exists an effective divisor $R\in
|-K_Y|$ passing through $P$ and a general point $Q\in \Delta_i$.
On the other hand, there is no $(-1)$-curve in $Y$ passing through
$Q$. So by Lemma \ref {claim} we may assume that $R$ is reduced and irreducible.
By Lemma \ref{lemma-e} $R$ is different from the components of
$D$. Assuming that $\mult_P(D)\ge 2$ we obtain
$$2=R\cdot D\ge \mult_P(D)+\delta_i>2,$$ a contradiction. This
proves the first assertion. Now the second follows since
$\delta_0>0$ in (\ref{begeq}).
\end{proof}

\begin{cor}\label{lc-klt}
The pair $(Y',D')$ is Kawamata log terminal in codimension one
and is not log canonical at some point $P'\in E'$.
\end{cor}

\begin{proof} This follows from Lemma \ref{corollary-floor-D} taking into
account that $D'$ is the crepant pull-back of $D$,
see \cite[L.\ 3.10]{Kollar-1995-pairs}.
\end{proof}

Since $\dim |-K_{Y'}|=1$
there exists an element $C'\in |-K_{Y'}|$
passing through the point $P'$ as in Corollary \ref{lc-klt}.

\begin{lem}\label{smooth}
The point $P\in Y$ is a smooth point of the
image $C=f_*C'$.
\end{lem}

\begin{proof}
This follows by Lemma \ref{222}
since $C\in |-K_Y|$ passes through $P$.
\end{proof}

\begin{cor}\label{corollary-E-component-C}
$E'$ is not a component of $C'$.
\end{cor}
\begin{proof}
We can write $f^*C=C'+k E'$ for some $k\in \Z$.
Then $k=- k E'^2= C'\cdot E'=1$.
 By Lemma  \ref{smooth} the coefficient of $E'$ in   $f^*C$ is
equal to $1$ as well. Now  the assertion follows.
\end{proof}

\begin{lem}\label{r}
$C$ is reducible.
\end{lem}

\begin{proof}
Indeed, otherwise $C'$ is irreducible by Corollary \ref{corollary-E-component-C}.
Since $\mult_{P'}D'>1$ by Corollary \ref{lc-klt} and $D'\cdot C'=K_{Y'}^2=1$,
$C'$ is a component of $D'$.
Hence $C$ is a component of $D$. This contradicts Lemma \ref{lemma-e}.
\end{proof}

\begin{lem}\label{225}
We have $C'=C_1'+C_2'$, where
$C_1$ is a $(-1)$-curve, $C_2'$ is a $(-2)$-curve, and
$C_1'\cdot C_2'=2$.
Furthermore,  $P'\in C_2'\setminus  C_1'$, and $C_2=f(C_2')$ is a
$(-1)$-curve.
\end{lem}

\begin{proof}
Since $C$ is reducible and $C\in |-K_Y|$, by Lemma \ref{claim}
$C=C_1+C_2$, where $C_1$, $C_2$ are $(-1)$-curves with $C_1\cdot
C_2=2$. By Lemma \ref{smooth} $P\notin C_1\cap C_2$,  where
$C_2$ is a component of $D$ by Corollary \ref{cor00121}, while by
Lemma \ref{lemma-e} $C_1$ is not. So we may assume that $P\in
C_2\setminus C_1$. Now the lemma follows from Corollary \ref{lc-klt}.
\end{proof}

\begin{sit}
\label{1226} Letting in the sequel $C_2=\Delta_1$ we can write
$D=\delta_1 C_2+D_{\res}$, where $\delta_1>0$, $D_{\res}$ is an
effective $\Q$-divisor, and $C_2$ is not a component of
$D_{\res}$. Similarly $$D'=\delta_1
C_2'+D_{\res}'+\delta_0 E',$$ where $D_{\res}'$ is the proper
transform of $D_{\res}$ and $\delta_0=\operatorname{mult}_P(D)-1$
(cf.\ (\ref{begeq})).
\end{sit}

\begin{lem}\label{227}
We have $2\delta_1\le 1$.
\end{lem}

\begin{proof} This follows from
\[
0\le D_{\res}\cdot C_1=(D-\delta_1 C_2)\cdot C_1=1-2\delta_1\,.
\]
\end{proof}

\begin{lem}\label{228} In the notation as before
$\delta_0+D_{\res}'\cdot C_2'>1$.
\end{lem}

\begin{proof}
Let us show first that $\{P'\}=C_2'\cap E'=C_2'\cap\supp(D'_{\res})$.
Indeed, $P'\in E'$ by construction, $P'\in C_2'$ by
Lemma \ref{225}, and $P'\in \supp(D'_{\res})$
because otherwise $P'$ would be a node of $D'$ (indeed, $E'$ meets $C_2'$
transversally at $P'$) and so the pair $(Y',D')$ would be log canonical at
$P'$ contrary to Corollary \ref{lc-klt}. On the other hand,  the curves
$C_2'$ and $D'_{\res}$ have only one point in common
by Lemma  \ref{lemma-degenerate-fibers}(i).

Since $\delta_1<1$ the pair $(Y', C_2'+D_{\res}'+\delta_0 E')$ is
not log canonical at $P'$. Now applying \cite[Corollary
5.57]{Kollar-Mori-19988} we obtain
\[
1< (D_{\res}'+\delta_0 E')\cdot C_2'= \delta_0+D_{\res}'\cdot C_2',
\]
as stated.
\end{proof}

\begin{proof}[Proof of Proposition \textup{\ref{yuta1}} in the case $d=2$.]
We use the notation as above. Since $C_2'$ is a
$(-2)$-curve, by virtue of Lemmas \ref{227} and \ref{228} we obtain
\[
1-\delta_0<D_{\res}'\cdot C_2'= (D'-\delta_1
C_2'-\delta_0 E') \cdot C_2'= 2\delta_1-\delta_0\le
1-\delta_0,
\]
a contradiction. Now the proof of
Proposition \ref{yuta1} is completed.
\end{proof}

\begin{rem}\label{che} Our proof of Proposition \ref{yuta1}
goes along the lines of
that of Lemmas 3.1 and 3.5 in \cite{Cheltsov-log-can}. \footnote{
We are grateful to Ivan Cheltsov who attracted our attention to
this fact.} However, this proposition does not follow immediately
from the results in \cite{Cheltsov-log-can}. Indeed, in notation
of \cite{Cheltsov-log-can}  by Lemma \ref{lem4} we have
$\operatorname{lct}(Y,D)<1$. This is not sufficient to get a
contradiction with \cite[Theorem 1.7]{Cheltsov-log-can}. The point
is that our boundary $D$ is not arbitrary, in contrary, it is
rather special (see Lemma \ref{lemma-degenerate-fibers}).
\end{rem}

\subsection*{ Acknowledgements. }
This work was done during a stay of the second and third authors
at the Max Planck Institute f\"ur Mathematik at Bonn and a stay of
the first and the second authors at the Institute Fourier,
Grenoble. The authors thank these institutions for their
hospitality and support.

\end{document}